\documentclass{amsart}
\usepackage{amsthm, amsfonts, amssymb}

  \gdef\rtitle{}

\gdef\rauthor{}

\AtBeginDocument{%
    \markboth{\rauthor, \rtitle\hfil}{\hfil\rauthor, \rtitle}%
    }

\newtheorem{theorem}{Theorem}[section]
\newtheorem{lemma}[theorem]{Lemma}
\newtheorem{corollary}[theorem]{Corollary}
\newtheorem{example}[theorem]{Example}

\newcommand{\m}{\frak{m}}

\newcommand{\reg}{\operatorname{reg}}
\newcommand{\projdim}{\operatorname{projdim}}

\newcommand{\Tor}{\text{Tor}}

\begin{document}

\title{Graded Betti numbers of ideals with linear quotients}
\author{Leila Sharifan and Matteo Varbaro}
\address{Department of Pure Mathematics,  Faculty of Mathematics and  Computer Science, Amirkabir University of Technology, 424, Hafez Ave.,  Tehran 15914, Iran}
\email{leila-sharifan@aut.ac.ir}
\address{Dipartimento di Matematica, Universit\'a  di Genova, Via Dodecaneso 35, I-16146 Genova, Italy}
\email{varbaro@dima.unige.it} \subjclass{}
\date{}
 \maketitle


\begin{abstract}
In this paper we show that every ideal with linear quotients is
componentwise linear. We also generalize the Eliahou-Kervaire
formula for graded Betti numbers of stable ideals to  homogeneous
ideals with linear quotients.
\end{abstract}

\markboth{\large Graded Betti numbers of ideals with linear quotients}{\large Leila Sharifan and Matteo Varbaro}

\section{Introduction}

Let $K$ be a field, $S=K[x_1,...,x_n]$ the polynomial ring in $n$
variables, and $I\subset S$ a homogeneous ideal. Let
$\{f_1,...,f_m\}$ be a system of homogeneous generators for $I$. We
say that $I$  has {\it linear quotients with respect to the elements
$f_1,...,f_m$} if the ideal $\langle f_1,...,f_{i-1} \rangle :f_i$
is generated by linear forms for all $i=2,\ldots ,m$ (notice that
this property depends on the order of the generators). Monomial
ideals with linear quotients were introduced in \cite{HT} by Herzog
and Takayama and have strong combinatorial implication, see the
paper of Soleyman Jahan and Zheng \cite{SZ}.

In this paper we study the minimal free resolution of a graded ideal $I$ with linear quotients
 with respect to a {\it minimal} system of homogeneous generators $\{f_1, \ldots , f_m\}$ (in this case we simply say
 that {\it $I$ has linear quotients}). It is known that $I$ is componentwise linear provided that $\deg(f_1)\leq ...\leq \deg(f_m)$, see the book  of Herzog and Hibi \cite{HH2} . We prove that $I$ is componentwise linear without the additional assumption on the degrees, giving an affirmative answer to a question of Herzog (see Corollary \ref{linear quotient}). Our result is a generalization of \cite[Corollary 2.8]{SZ}, where the authors prove componentwise linearity in  case  $I$ is a monomial ideal and $f_1, \ldots , f_m$ are monomials. \\
A large class of monomial ideals with linear quotients is the class
of stable ideals. A minimal free resolution of stable ideals  was
constructed by Eliahou and Kervaire in \cite{EK}, who also give  an
explicit formula for their graded Betti numbers. For an ideal with
linear quotients we give a formula for its Betti numbers that
generalizes the formula by Eliahou and Kervaire. We express the graded Betti numbers in terms of the numbers $\deg(f_i)$ and the minimal number of generators of $\langle f_1, \ldots, f_{i-1}\rangle :f_i$, for $i=1, \ldots ,m$. As a consequence we will also obtain the Castelnuovo-Mumford regularity and the projective dimension of $I$. \\

\section{Graded Ideals With Linear Quotients}

Let $I \subset S$ be a graded ideal  and  $\{f_1,...,f_m\}$ be a minimal system of homogeneous generators for $I$ such that $I$  has linear quotients with respect to $f_1,...,f_m$.
Notice that if an ideal has linear quotients with respect to a minimal system of homogeneous generators then it need not have linear quotients with respect to all minimal homogeneous system of generators (see Example \ref{syst gen}).\\
It is easy to see that $\deg(f_i )\geq \min\{\deg(f_1), . . .
,\deg(f_{i-1})\}$ for $i = 2, . . . ,m$. Particularly, $\deg(f_1)
\leq \deg(f_i)$ for every $i=1, \ldots ,m$. But in general, the sequence $\deg(f_1),...,\deg(f_{i-1})$ of degrees need not be increasing.\\
For example, the ideal $\langle xy, xy^3z + y^4z - y^3z^2, x^3 +
x^2y - x^2z, x^2z^3\rangle$ has linear quotients in the given order,
but $\deg(xy^3z + y^4z - y^3z^2)>\deg(x^3 + x^2y - x^2z)$.

\begin{example}\label{syst gen}
Set $\m=\langle x_1, \ldots ,x_n \rangle \subset S$ the maximal irrelevant ideal. It is easy to see that $\m^k$ has linear quotients with respect to its unique system of monomial generators for every natural number $k$ (for example ordering the monomials lexicographically).\\
If $K$ is large enough, we can choose a minimal system of
homogeneous generators $\{ g_1, \ldots, g_t\}$ for $\m^k$ such that
$\operatorname{GCD}(g_i,g_j)=1$ for each $i \neq j$. Hence it is
clear that, if $k\geq 2$, $\m^k$ does not have linear quotients with
respect to $\{g_1, \ldots ,g_t\}$, for any order, otherwise the
unique factorization would be contradicted.
\end{example}

If $J$ is a graded ideal of $S$, then we write $J_{\langle
j\rangle}$ for the ideal generated by all homogeneous polynomials of
degree $j$ belonging to $J$. Moreover, we write $J_{\leq k}$ for the
ideal generated by all homogeneous polynomials in $J$ whose degree
is less than or equal to $k$. We will say that $J$ is componentwise
linear if $J_{\langle j\rangle}$ has a $j$-linear resolution for all
$j$.

 In \cite[Lemma 2.1]{SZ} the authors show that for any monomial ideal $J$ with linear quotients with respect to
the unique minimal system $G(J)$ of monomial generators of $J$,
there exists a degree increasing
 order of the elements $u_1, . . . ,u_m$ of $G(J)$ such that $J$  has linear quotients with respect to this order. As a corollary they are able to prove that $J$ is componentwise linear. It is worth remarking that if $J$ has linear quotients with respect to a non-minimal system of generators, then $J$ is not necessarily componentwise linear (see Example \ref{not minimal}).
We prove the result of \cite{SZ} for an arbitrary graded ideal with
linear quotients using completely different methods.

In the following for a $\mathbb{Z}$-graded module $M$ we write $M(a)$ for the module
obtained by {\it shifting degrees} by $a$, for any $a \in \mathbb{Z}$: i.e. $M(a)$ is the module $M$ with the grading defined as $M(a)_b=M_{a+b}$.

\vspace {.2cm}

 For our purpose we need the next lemma.

 \begin{lemma} \label{main lemma}
Let $J=\langle f_1,...,f_m\rangle$ be a componentwise linear  ideal and let $I=J+\langle f\rangle$ where $f$ is a homogenous form of degree $d$. If the ideal $J:f$ is generated by linear forms then  $I_{\langle j\rangle}$ has  a $j$-linear resolution for each $j\leq d$ provided that $f_1,...,f_m,f$ is a minimal system of generators for $I$.
 \end{lemma}
\begin{proof}
$I_{\langle j\rangle}=J_{\langle j\rangle}$ for each $j<d$ so it is
enough to prove that $I_{\langle d\rangle}$ has $d$-linear
resolution. First  we show that  $J:f=J_{\langle d\rangle}:f$. It is
clear that $J_{\langle d\rangle}\subseteq J_{\leq d}\subseteq J$.
Therefore $J_{\langle d\rangle}:f\subseteq J_{\leq d }:f\subseteq
J:f$. Set $L=J:f=\langle l_1,...,l_r\rangle$ where $l_i$ are linear
forms. So there exist non-zero homogeneous polynomials
$\lambda_{t_{1i}},...,\lambda_{t_{s_ii}}$ such that
$l_if=\sum_{k=1}^{s_i}\lambda_{t_{ki}}f_{t_{ki}}$ and
$\deg(\lambda_{t_{ki}})+\deg(f_{t_{ki}})=\deg(l_if)=d+1$. Since
$f_1,...,f_m,f$ is a minimal system of generators for $I$,
$\deg(\lambda_{t_{ki}})\geq 1$ and so $\deg(f_{t_{ki}})\leq d$. Then
$l_if\in J_{\leq d}$ but it is clear that $(J_{\langle d\rangle})_{
d+1}=(J_{\leq d})_{d+1}$. So $l_if\in J_{\langle d\rangle}$ and
$J:f=J_{\langle d\rangle}:f$.

Trivially $J_{\langle d\rangle}+\langle f\rangle=I_{\langle d\rangle}$, so by the above discussion
\[ I_{\langle d\rangle}/J_{\langle d\rangle} \cong (S/(J_{\langle d\rangle}:f))(-d)=(S/(J:f))(-d)=S/L(-d).\]
Hence we can consider the short  exact sequence
$$0\to J_{\langle d\rangle}\to I_{\langle d\rangle}\to S/L(-d)\to 0$$
which  yields the long exact sequence of Tor
$$...\to \Tor_i^S(K,J_{\langle d\rangle})_{i+j}\to \Tor_i^S(K,I_{\langle d\rangle})_{i+j}\to \Tor_{i}^S(K,S/L(-d))_{i+j}\to ....$$
The Tor-groups on the right and on the left end of this sequence
vanish for $j \neq d$ since   the corresponding modules have
$d$-linear resolution. This proves that $I_{\langle d\rangle}$ has
$d$-linear resolution.
\end{proof}

We are ready to prove the main theorem of the paper.

\begin{theorem} \label{componentwise}
Let $J=\langle f_1,...,f_m\rangle$ be a componentwise linear  ideal and let $I=J+\langle f\rangle$ where $f$ is a homogeneous form of degree $d$. If the ideal $J:f$ is generated by linear forms then  $I$ is componentwise linear provided that $f_1,...,f_m,f$ is a minimal system of generators for $I$.
\end{theorem}
\begin{proof}
 By Lemma \ref{main lemma}, $I_{\langle j\rangle}$ has $j$-linear resolution for each $j\leq d$. Set $\m=\langle x_1,...,x_
n\rangle$ and  $p=\max\{\deg(f_i)\}$.\\ if $d>p$ then   $I_{\langle d+j\rangle}=\m^j I_{\langle d\rangle}$ for each $j\geq1$. So the exact sequence
\[0\to  I_{\langle d+i\rangle}\to I_{\langle d+i-1\rangle}\to I_{\langle d+i-1\rangle}/ I_{\langle d+i\rangle} \to 0 \]
 and induction on $i$ shows that $I_{\langle d+i\rangle}$ has $d+i$-linear resolution. Therefore, $I$ is componentwise linear ideal.\\ If $d\leq p$, let $s=p-d$. In this case we use induction on $s$ to prove the result. If $s=0$, again by Lemma \ref{main lemma}, $I$ is componentwise linear arguing as above. Now suppose that $s\geq 1$ and  the result is true for $s-1$. Let $L=J:f=\langle l_1,...,l_r\rangle$ minimally generated by $r$ linear forms. If $r=n$ then $\m f\subseteq J$. So $J_{\langle d+i\rangle}=I_{\langle d+i\rangle}$ for each $i\geq1$. Now another application of Lemma \ref{main lemma} shows that  $I$ is a componentwise linear ideal. If $r<n$ we complete $\{l_1,...,l_r\}$ to a $K$-basis of $S_1$ by new linear forms $e_1,...,e_{n-r}$. Set $g_i=e_if$ for $i=1,...,n-r$.  One can easily check that $(J+\langle g_1,...,g_{n-r}\rangle)_{\langle d+i\rangle}=I_{\langle d+i\rangle}$ for each $i\geq 1$. Therefore it is enough to show that $J+\langle g_1,...,g_{n-r}\rangle$ is componentwise linear.

{\bf Claim 1:} For every  $i=1,...,n-r$ we have $(J+\langle g_1,...,g_{i-1}\rangle):g_i=L+\langle e_1,...,e_{i-1}\rangle$.\\
{\it Proof of Claim 1}: Let $h\in (J+\langle
g_1,...,g_{i-1}\rangle):g_i$. So $he_if\in J+\langle
g_1,...,g_{i-1}\rangle$. Therefore there exist suitable coefficients
$\gamma_j$ and $\lambda_j$ such that
$he_if=\sum_{j=1}^m\gamma_jf_j+\sum_{j=1}^{i-1}\lambda_je_jf$. Then
$he_i-\sum_{j=1}^{i-1}\lambda_je_j\in J:f$. So $he_i\in\langle
l_1,...,l_r,e_1,...,e_{i-1}\rangle$. Since
$l_1,...,l_r,e_1,...e_{n-r}$ is a regular sequence $h\in L+\langle
e_1,...,e_{i-1}\rangle$. The other inclusion is clear.

 {\bf Claim 2:} For every  $i=1,...,n-r$ the ideal $J+\langle g_1,...,g_{i}\rangle$ is minimally generated by $\{f_1,...,f_m,g_1,...,g_i\}$.\\
{\it Proof of Claim 2}: Suppose that for some $t$, $f_t\in\langle
f_1,...,\widehat f_t,...,f_m,g_1,...,g_i\rangle$. Then there exist
suitable coefficients $\gamma_j$ and $\lambda_j$ such that
$f_t={\displaystyle{\sum^{m}_{{j=1} \atop {j \neq
t}}}}\gamma_jf_j+\sum_{j=1}^{i}\lambda_jg_j={\displaystyle{\sum^{m}_{{j=1}
\atop {j \neq t}}}}\gamma_jf_j+(\sum_{j=1}^{i}\lambda_je_j)f$. So
$\{f_1,...,f_m,f\}$ is not a minimal system of generators for $I$
which is a contradiction.\\ If for some $t$, $g_t\in\langle
f_1,...,f_m,g_1,...,\widehat g_t,...,g_i\rangle$  then again one can
find coefficients $\gamma_j$ and $\lambda_j$ such that
$g_t=\sum_{j=1}^m\gamma_jf_j+{\displaystyle{\sum^{i}_{{j=1} \atop {j
\neq t}}}}\lambda_jg_j$. So $(e_t-{\displaystyle{\sum^{i}_{{j=1}
\atop {j \neq t}}}}\lambda_je_j)f\in J$ and therefore
$e_t-{\displaystyle{\sum^{i}_{{j=1} \atop {j \neq
t}}}}\lambda_je_j\in L$, which is a contradiction.

Since  $\deg(g_i)=d+1$, Claim 1 and Claim 2 together with the
inductive hypothesis show that $J+\langle g_1,...,g_{i}\rangle$ is
componentwise linear for each $i=1,...,n-r$. This completes the
proof.

\end{proof}

From the above theorem   the answer to a question  by Herzog
follows.

\begin{corollary}\label{linear quotient}
Let $I$ be a homogeneous ideal and  $\{f_1,...,f_m\}$ be a minimal
system of generators for $I$. If $I$  has linear quotients with
respect to  $f_1,...,f_m$ then $I$ is a componentwise linear ideal.

\end{corollary}
\begin{proof}
Set $I_i=\langle f_1,...,f_i\rangle$ for $i=1, \ldots ,m$. It is
clear that $I_1$ has a linear resolution. In particular $I_1$ is a
componentwise linear ideal. Therefore, one can apply Theorem
\ref{componentwise} and induction on $i$ to see that each $I_i$ is
componentwise linear for $i=1, \ldots ,m$.
\end{proof}

\begin{example}\label{not minimal}
In Corollary \ref{linear quotient} one cannot remove the hypothesis that $I$ is minimally generated by $\{f_1,\ldots ,f_m\}$.
For instance, let $I$ be the monomial ideal $I=\langle x^2,y^2 \rangle \subset K[x,y]$. Trivially $I=I_{\langle 2 \rangle}$ does not have a $2$-linear resolution, but it has linear quotients with respect to the system of generators $\{x^2,xy^2,y^2\}$.\\
In general we can observe that every $\m$-primary ideal $I\subset
S$, where $\m=\langle x_1, \ldots ,x_n \rangle$, has linear
quotients with respect to a homogeneous system of generators. In
fact  let $h$ be such that $\m^h \subset I$. Then clearly by
ordering the elements of $\m^h$ appropriately  (for example
lexicographically), $I$ has linear quotients with respect to
$\{\m^h, I_{h-1}, \ldots, I_s\}$, where $s$ is minimal such that
$I_s \neq 0$.
\end{example}

In the following we compute the graded Betti numbers of ideals with
linear quotients. For this, we use the formula (1) of Herzog and
Hibi (see \cite[Proposition 1.3]{HH}) for the graded Betti numbers
of an arbitrary componentwise linear ideal $J$.
$$\beta_{i,i+j}(J) =\beta_i(J_{\langle j\rangle})- \beta_i(\m J_{\langle j-1\rangle}) \hspace{.8cm} {\text {for all}}\hspace{.2cm}i, j \eqno{(1)}$$

\begin{theorem}\label{Betti}
Let $J=\langle f_1,...,f_m\rangle$ be a componentwise linear  ideal and let $I=J+\langle f\rangle$ where $f$ is a homogeneous form of degree $d$. If the ideal $J:f$ is minimally  generated by $r$ linear forms then
$$\beta_{i,i+j}(I)=\beta_{i,i+j}(J)\hspace{.8cm} {\text if}\hspace{.2cm} j\neq d, $$
$$\beta_{i,i+d}(I)=\beta_{i,i+d}(J)+{r\choose i},$$
provided that $\{f_1,...,f_m,f\}$ is a minimal system of homogeneous generators for $I$.

\end{theorem}
\begin{proof}
Consider the exact sequence
$$ 0\to J\to I\to S/L(-d)\to 0 \eqno{(2)} $$
which yields the long exact sequence
 \begin{eqnarray*}
... & \to \Tor_{i+1}^S(K,S/L(-d))_{i+j} & \to  \Tor_i^S(K,J)_{i+j}  \to \\
&& \to \Tor_i^S(K,I)_{i+j}  \to  \Tor_i^S(K,S/L(-d))_{i+j}  \to  ...
\end{eqnarray*}

The Tor-groups on the right and on the left
end of this sequence vanish for $j \neq d,d+1$ since   the
corresponding module has a $d$-linear resolution. This proves that $\beta_{i,i+j}(I)=\beta_{i,i+j}(J)$ if $j\neq d,d+1$.

First consider $j=d$. As discussed in the proof of Lemma \ref{main
lemma}, one has the short exact sequence
$$0\to J_{\langle d\rangle}\to I_{\langle d\rangle}\to S/L(-d)\to 0$$
which  yields the long exact sequence
$$...\to \Tor_i^S(K,J_{\langle d\rangle})_{i+j}\to \Tor_i^S(K,I_{\langle d\rangle})_{i+j}\to \Tor_{i}^S(K,S/L(-d))_{i+j}\to ...$$
Since both $J_{\langle d \rangle}$ and $S/L(-d)$ have a $d$-linear
resolution, the above long exact sequence gives
\[\beta_i(I_{\langle d \rangle})= \beta_i(J_{\langle d \rangle})+\beta_i(S/L(-d))\]
Moreover the minimal free resolution of $S/L(-d)$ is given by Koszul complex, so
$$\beta_i(S/L(-d))={r\choose i}$$
Thus by (1),
 \begin{eqnarray*}
 \beta_{i,i+d}(I)&=&\beta_i(I_{\langle d\rangle})-\beta_i(\m I_{\langle d-1\rangle})\\
 & = &\beta_i(I_{\langle d\rangle})-\beta_i(\m J_{\langle d-1\rangle})\\
 & = &\beta_i(J_{\langle d\rangle})+{r\choose i}-\beta_i(\m J_{\langle d-1\rangle})\\
 & = &\beta_{i,i+d}(J)+{r\choose i}.
\end{eqnarray*}

Now consider $j=d+1$. By (2) we have the following exact sequence
\begin{eqnarray*}
&0&\to\Tor_{i+1}^S(K,J)_{i+d+1}  \to\Tor_{i+1}^S(K,I)_{i+d+1}\to\Tor_{i+1}^S(K,S/L(-d))_{i+d+1}\\
&&\to\Tor_{i}^S(K,J)_{i+d+1}\to\Tor_{i}^S(K,I)_{i+d+1}\to 0.
\end{eqnarray*}
So by the first part
\begin{eqnarray*}
\beta_{i,i+d+1}(I) & = & \beta_{i,i+d+1}(J)-{r\choose i+1}+\beta_{i+1,i+d+1}(J)+{r\choose i+1}-\beta_{i+1,i+d+1}(J)\\
& = & \beta_{i,i+d+1}(J).
\end{eqnarray*}
\end{proof}

\begin{corollary}\label{Betti linear}
Let $I$ be a homogeneous ideal with linear quotients with respect to   $f_1,...,f_m$  where $\{f_1,...,f_m\}$ is a
 minimal system of homogeneous generators for $I$. Let $n_p$ be the minimal number of homogeneous generators of $\langle f_1,..,f_{p-1}\rangle:f_p$ for $p=1,...,m$.
 Then
\[ \reg(I)=\max \{ \deg (f_p):p=1, \ldots ,m\}  \]
$$\projdim(I)=\max\{n_p|1\leq p\leq m\}$$
 $$\beta_{i,i+j}(I)=\sum_{1\leq p\leq m,\hspace{.2cm} \deg(f_p)=j}{n_p\choose i}$$
$$\beta_i(I)=\sum_{p=1}^m{n_p\choose i}.$$
\end{corollary}
\begin{proof}
It suffices to compute the graded Betti numbers of $I$. Let
$I_p=\langle f_1,...,f_p\rangle$. Using Corollary \ref{linear
quotient} and Theorem \ref{Betti} it is easy to find the Betti
numbers of each $I_p$ with inductive process on $p$.
\end{proof}

A stable ideal is a monomial ideal $J \subset S=K[x_1, \ldots, x_n]$ with the following property: if $u$ is a monomial belonging to $J$, $m(u):=\max \{k:x_k | u\}$ and $i<m(u)$, then $x_i(u/x_{m(u)})\in J$.\\  Suppose that $J$ is a stable ideal and the minimal monomial generators $u_1,...,u_m$ of $J$ are ordered so that  $i < q$ if and only if either the degree of $u_i$ is less than the degree of $u_q$, or the degrees are equal and
$u_i>_{\text{revlex}}u_q$. One can easily see that for $p\geq 1$ we have
$$\langle u_1, . . . ,u_{p-1}\rangle : u_p = \langle x_1,..., x_{m(u_p)-1} \rangle.$$
In particular $J$ has linear quotients.\\
The minimal free resolution of a stable ideal is given by the Eliahou-Kervaire resolution \cite{EK}. From this resolution one can easily find  the  regularity, the projective dimension and the Betti numbers of a stable ideal. The Eliahou-Kervaire formula for the graded Betti numbers of a stable ideal $J$ as above is the following:
$$\beta_{i,i+j}(J)=\sum_{1\leq p\leq m,\hspace{.2cm} \deg(u_p)=j}{m(u_p)-1\choose i}$$
From the above discussion  it is clear that Corollary \ref{Betti
linear} generalizes the Eliahou-Kervaire formula to all the
homogenous ideals with linear quotients.

\vspace{5mm}

{\it Acknowledgments.}
The heart of this paper was born during the PRAGMATIC school which took place in Catania during the summer of 2008.\\
So we would like to thank our teachers J\"{u}rgen Herzog and Volkmar Welker, who let us know beautiful aspects of  mathematics and many interesting open problems.\\
We are grateful to organizers of PRAGMATIC, especially to Alfio Ragusa and Giuseppe Zappal\`a for organizing the school.\\
Finally we also thank Aldo Conca for valuable
discussions about the topics of the paper.

 \end{document}